\documentclass[11pt,reqno,a4paper]{amsart}
\usepackage{amssymb}
\usepackage{amsmath}
\usepackage[mathscr]{euscript}
\usepackage{hyperref}
\usepackage{graphicx}
\usepackage{mathabx}
\usepackage{comment}
\usepackage{amssymb}
\usepackage{enumitem}

\theoremstyle{plain}
\numberwithin{equation}{section}
\newtheorem{theorem}{Theorem}
\newtheorem{proposition}{Proposition}

\newtheorem{definition}{Definition}

\theoremstyle{remark}
\newtheorem{remark}{Remark}
\newtheorem{example}{Example}
\renewcommand{\epsilon}{\varepsilon}
\renewcommand{\phi}{\varphi}

\DeclareMathOperator{\Ker}{Ker}

\def\R{\mathbb{R}}

\def\Id{\text{\rm Id}}

\usepackage{xcolor}
\usepackage{amsthm}
\usepackage{amsmath}
\usepackage{amsfonts}

\begin{document}
\title[Robustness of nonuniform exponential dichotomies]{Robustness of nonuniform exponential dichotomies under a wider class of perturbations}

\author[D. Dragi\v cevi\'c]{Davor Dragi\v cevi\'c}
\address{Faculty of Mathematics, University of Rijeka, Radmile Matej\v ci\' c 2, 51000 Rijeka, Croatia}
\email{ddragicevic@math.uniri.hr}

\begin{abstract}
The robustness property of exponential dichotomies refers to the stability of this notion under small linear perturbations. In recent work~\cite{PPX}, the authors have identified a new class of perturbations under which the notion of a nonuniform exponential dichotomy persists. In the present paper, we show that it is possible to extend this class. Moreover, unlike~\cite{PPX} where the results are restricted to the case of ordinary differential equations, in the present paper we deal with arbitrary evolution families consisting of possibly noninvertible linear operators. 
\end{abstract}

\maketitle

\section{Introduction}
The notion of a (uniform) exponential dichotomy introduced by Perron~\cite{Perron} plays an important role in the qualitative theory of nonautonomous dynamical systems. It can be regarded as a nonautonomous counterpart to the classical notion of hyperbolicity for autonomous dynamics. It has many interesting consequences including the existence of stable and unstable invariant manifolds for any sufficiently small nonlinear perturbation of a linear dynamics exhibiting exponential dichotomy (see, for example,~\cite{Aulbach}). In the same context, it is possible to obtain a  nonautonomous version of the classical Grobman-Hartman linearization result (see~\cite{Palmer}).

An important feature of the notion of an exponential dichotomy is its robustness property, which asserts that this notion persists under sufficiently small linear perturbations. This property was discussed by Massera and Sch\"affer~\cite{MS} as well as Coppel~\cite{Coppel} for nonautonomous ordinary differential equations on a finite-dimensional space. The case of ordinary differential equations on Banach spaces was treated by Dalecki\u \i \  and Kre\u \i n~\cite{DK}. A major contribution was due to Henry~\cite{Henry} who established the robustness property of exponential dichotomies for arbitrary evolution families consisting of possibly noninvertible linear operators, making the results applicable to the study of partial differential equations. For more recent contributions dealing with the robustness property of uniform exponential behavior for various classes of nonautonomous dynamics, we refer to~\cite{DSS, Huy, NP, PS, Popescu, Sasu} and references therein.

Recently, Barreira and Valls introduced the notion of a \emph{nonuniform} exponential dichotomy~\cite{BVbook}, which includes the notion of a uniform exponential dichotomy as a particular case. It turns out that the notion of a nonuniform exponential dichotomy is ubiquitous from the ergodic theory point of view (see~\cite[Chapter 10]{BVbook}). 

The robustness property of nonuniform exponential dichotomies for difference equations and ordinary differential equations on Banach spaces was established in~\cite{BSV} and~\cite{BV}, respectively.  Easier proofs based on the relationship between the existence of nonuniform exponential dichotomies and the so-called admissibility were given in~\cite{BDV1, BDV3, BDV2}.

In their recent paper~\cite{PPX}, Pinto, Poblete and Xia proposed a new condition for the robustness of nonuniform exponential dichotomies for nonautonomous dynamics with continuous time. In order to describe their idea, we will restrict (for simplicity) to the particular case of uniform exponential dichotomies. If a nonautonomous linear ordinary differential equation on a Banach space
\begin{equation}\label{LDE-intro}
x'=A(t)x \quad t\in \R
\end{equation}
admits a uniform exponential dichotomy, then it is well known (see~\cite{Popescu}) that its perturbation
\[
x'=(A(t)+B(t))x \quad t\in \R
\]
admits a uniform exponential dichotomy, provided that 
\begin{equation}\label{smallness}
\sup_{t\in \R}\|B(t)\|\le c,
\end{equation}
where $c>0$ is sufficiently small. We stress that the smallness of $c$ can be quantified in terms of constants related to the dichotomy of~\eqref{LDE-intro}. The main objective of~\cite{PPX} is to replace~\eqref{smallness} with a weaker condition, which, roughly speaking, replaces the uniform smallness condition~\eqref{smallness} with a smallness integrability condition (see Remark~\ref{RM3} for details).

The main purpose of this note is to illustrate that the robustness condition given in~\cite{PPX} can be further relaxed for a wide class of nonuniform exponential dichotomies. Moreover, in sharp contrast to~\cite{PPX}, we show that this can be done without making any restrictions on the size of the nonuniform part of the dichotomy (constant $\varepsilon$ in~\eqref{D1} and~\eqref{D2}).  We refer to Remark~\ref{RM3} for a detailed discussion.
Furthermore, while the results from~\cite{PPX} are restricted to the case of ordinary differential equations, in the present paper, we consider the case of possibly noninvertible dynamics by considering arbitrary evolution families.   

We stress that we are able to achieve all this and simultaneously simplify the arguments from~\cite{PPX} relying on the relationship mentioned above between exponential dichotomies and the admissibility property.

\section{Preliminaries}
Throughout this note $X=(X, \|\cdot \|)$ will denote an arbitrary Banach space. Moreover, $\mathcal B(X)$ will denote the space of all bounded linear operators acting on $X$ equipped with the operator norm that we also denote by $\|\cdot \|$.

We recall that a \emph{evolution family} on $X$ over $\mathbb R$ is a family $\{T(t, s): \ t, s\in \R, \ t\ge s\}\subset \mathcal B(X)$ satisfying the following properties:
\begin{itemize}
\item $T(t, t)=\Id$ for each $t\in \R$, where $\Id$ denotes the identity operator on $X$;
\item for $t\ge s\ge r$, 
\[
T(t, s)T(s, r)=T(t, r);
\]
\item for $x\in X$ and $s\in \R$, the map $t\mapsto T(t, s)x$ is continuous on $[s, \infty)$. Moreover,  for $x\in X$ and $t\in \R$, the map $s\mapsto T(t, s)x$ is continuous on $(-\infty, t]$.
\end{itemize}

We also recall the notion of a nonuniform exponential dichotomy.
\begin{definition}\label{NED}
We say that an evolution family $\{T(t, s): \ t, s\in \R, \ t\ge s\}\subset \mathcal   B(X)$ admits a \emph{nonuniform exponential dichotomy} if there exist constants $K, \alpha>0$ and $\varepsilon \ge 0$ and a family $\{P(t): t\in \mathbb R\}\subset \mathcal B(X)$ of projections on $X$ such that the following holds:
\begin{enumerate}
\item for $t\ge s$,
\[
T(t, s)P(s)=P(t)T(t, s),
\]
and $T(t, s)\rvert_{\Ker P(s)}\colon \Ker P(s)\to \Ker P(t)$ is an isomorphism;
\item for $t\ge s$,
\begin{equation}\label{D1}
\|T(t, s)P(s)\| \le Ke^{-\alpha (t-s)+\varepsilon |s|};
\end{equation}
\item for $t\le s$,
\begin{equation}\label{D2}
\|T(t, s)(\Id-P(s))\| \le Ke^{-\alpha (s-t)+\varepsilon |s|},
\end{equation}
where 
\[
T(t, s):=\left (T(s, t)\rvert_{\Ker P(t)}\right )^{-1}\colon \Ker P(s)\to \Ker P(t).
\]
\end{enumerate}
\end{definition}
The following result is obtained in~\cite[Proposition 5.6]{BDVbook}.
\begin{proposition}\label{LN}
Let $\{T(t, s): \ t, s\in \R, \ t\ge s\}\subset \mathcal B(X)$ be an evolution family that admits a nonuniform exponential dichotomy. Furthermore, let $K, \alpha>0$, $\varepsilon \ge 0$, and $\{P(t): t\in \R\}$ be as in Definition~\ref{NED}.
Then there exists a family $\{\| \cdot \|_t: \ t\in \mathbb R\}$ of norms on $X$ that satisfy the following properties:
\begin{enumerate}
\item for $t\in \R$ and $v\in X$,
\begin{equation}\label{ln1}
\|v\|\le \|v\|_t \le 2Ke^{\varepsilon|t|}\|v\|;
\end{equation}
\item for $t\ge s$ and $v\in X$,
\begin{equation}\label{ln2}
\|T(t, s)v\|_t \le e^{-\alpha (t-s)}\|v\|_s;
\end{equation}
\item for $t\le s$ and $v\in X$,
\begin{equation}\label{ln3}
\|T(t,s)v\|_t \le e^{-\alpha (s-t)}\|v\|_s;
\end{equation}
\item for each $v\in X$, the map $t\mapsto \|v\|_t$ is continuous.
\end{enumerate}
\end{proposition}
\begin{remark}\label{rem}
\begin{enumerate}
\item[(i)] For the convenience of the reader we recall the construction from~\cite{BDVbook}: for $v\in X$ and $t\in \R$ set
\begin{equation}\label{lnc}
\|v\|_t:=\sup_{s\ge t}\left (\|T(s, t)P(t)v\|e^{\alpha (s-t)}\right)+\sup_{s\le t}\left (\|T(s, t)Q(t)v\|e^{\alpha (t-s)}\right ),
\end{equation}
where $Q(t):=\Id-P(t)$.
\item[(ii)] We note that in the case where $P(t)=\Id$ for each $t\in \R$, or in the case where $P(t)=0$ for every $t\in \R$, we can replace $2K$ by $K$ in~\eqref{ln1} (as one of the two terms on the right-hand side of~\eqref{lnc} vanishes).
\end{enumerate}
\end{remark}

\section{The robustness result}
Throughout this section,  $\{T(t, s): \ t, s\in \R, \ t\ge s\}\subset \mathcal B(X)$ will be an evolution family. Moreover, let $B\colon \R\to \mathcal  B(X)$ be a continuous map. In the sequel, we suppose that there is an evolution family $\{U(t, s): \ t, s\in \R, \ t\ge s\}\subset \mathcal B(X)$ such that 
\begin{equation}\label{U}
U(t, s)=T(t, s)+\int_s^t T(t, \tau)B(\tau)U(\tau, s)\, d\tau,
\end{equation}
for $t\ge s$. We refer to~\cite{RSRV} for the sufficient conditions that guarantee this.

\begin{theorem}\label{robthm}
Suppose that $\{T(t, s): \ t, s\in \R, \ t\ge s\}\subset \mathcal  B(X)$ admits a nonuniform exponential dichotomy and let $K, \alpha>0$ and $\varepsilon\ge 0$ be as in Definition~\ref{NED}. Furthermore, assume that 
\begin{equation}\label{rob}
    q:=2K\sup_{t\in \R}\int_{-\infty}^\infty e^{-\alpha |t-s|}b(s)\, ds<1,
\end{equation}
where $b(t):=\|B(t)\|e^{\varepsilon |t|}$, $t\in \R$. Then the evolution family $\{U(t, s): \ t, s\in \R, \ t\ge s\}\subset \mathcal B(X)$ admits a nonuniform exponential dichotomy.
\end{theorem}

\begin{proof}
Let $\|\cdot\|_t$, $t\in \R$ be the family of norms given by Proposition~\ref{LN} (see Remark~\ref{rem}) associated with the nonuniform exponential dichotomy of the evolution family $T(t, s)$. We introduce two function spaces. More precisely, let $\check{C}$ denote the space of all continuous maps $x\colon \R \to X$ such that
\[
\|x\|_\infty:=\sup_{t\in \R}\|x(t)\|_t<+\infty.
\]
Then $(\check C, \|\cdot \|_\infty)$ is a Banach space. Moreover, let $\check{M}$ be the space of all locally (Bochner) integrable functions $x\colon \R \to X$ satisfying
\[
\|x\|_{\check M}:=\sup_{t\in \mathbb R}\int_t^{t+1}\|x(s)\|_s\, ds<+\infty.
\]
Then $(\check{M}, \|\cdot \|_{\check M})$ is a Banach space. 
By~\cite[Theorem 5.8]{BDVbook}, for each $y\in \check M$ there exists a unique $x\in \check C$ such that 
\begin{equation}\label{adm}
x(t)=T(t,s)x(s)+\int_s^t T(t, \tau)y(\tau)\, d\tau \quad \text{for $t\ge s$.}
\end{equation}
Moreover, $x$ is given by 
\begin{equation}\label{adm2}
x(t)=\int_{-\infty}^\infty \mathcal G(t,s)y(s)\, ds \quad t\in \R, 
\end{equation}
where 
\[
\mathcal G(t, s):=\begin{cases}
T(t, s)P(s) &t\ge s; \\
-T(t, s)(\Id-P(s)) & t<s.
\end{cases}
\]

Let $R\colon \mathcal D(R)\subset \check C\to \check M$ be a linear operator defined by $Rx=y$ on the domain $\mathcal D(R)$ formed by all $x\in \check C$ for which there exists $y\in \check M$ such that~\eqref{adm} holds.
Moreover, let $S\colon \mathcal D(S)\subset \check C\to \check M$ be defined in a same manner by replacing $T(t, s)$ with $U(t, s)$. We then have (see the proof of~\cite[Theorem 6.3]{BDVbook}) that $\mathcal D(R)=\mathcal D(S)$ and 
\begin{equation}\label{RS}
(Rx)(t)=(S x)(t)+B(t)x(t), \quad \text{for $t\in \R$ and $x\in \mathcal D(R)=\mathcal D(S)$.}
\end{equation}

Take an arbitrary $y\in \check M$ and define $\mathcal T\colon \check C\to \check C$ by 
\[
(\mathcal Tx)(t):=\int_{-\infty}^\infty \mathcal G(t, s)(y(s)+B(s)x(s))\, ds, \quad \text{for $t\in \R$ and $x\in \check C$.}
\]
In order to show that $\mathcal T$ is well defined we begin by 
observing that~\eqref{ln2} and~\eqref{ln3} imply that 
\begin{equation}\label{Green}
\|\mathcal G(t, s)v\|_t\le e^{-\alpha |t-s|}\|v\|_s, \quad \text{for $t, s\in \R$ and $v\in X$.}
\end{equation}
Hence, by using~\eqref{ln1} we have that for any $x\in \check C$,
\[
\begin{split}
&\|(\mathcal Tx)(t)\|_t \\
&\le \int_{-\infty}^\infty \|\mathcal G(t, s)(y(s)+B(s)x(s))\|_t\, ds \\
&\le \int_{-\infty}^\infty e^{-\alpha |t-s|}\|y(s)+B(s)x(s)\|_s\, ds \\
&\le \int_{-\infty}^\infty e^{-\alpha |t-s|}\|y(s)\|_s\, ds+2K\int_{-\infty}^\infty e^{-\alpha |t-s|}e^{\varepsilon |s|}\|B(s) x(s)\|\, ds \\
&\le \sum_{m=-\infty}^\infty \int_{t+m}^{t+m+1}e^{-\alpha |t-s|}\|y(s)\|_s\, ds+2K\int_{-\infty}^\infty e^{-\alpha |t-s|}e^{\varepsilon |s|}\|B(s)\| \|x(s)\|\, ds \\
&\le \frac{2D}{1-e^{-\lambda}}\|y\|_{\check M}+2K\int_{-\infty}^\infty e^{-\alpha |t-s|}b(s) \|x(s)\|_s\, ds\\
&\le \frac{2D}{1-e^{-\lambda}}\|y\|_{\check M}+q\|x\|_\infty,
\end{split}
\]
for all $t\in \R$. Consequently, 
\[
\|\mathcal Tx \|_\infty \le \frac{2D}{1-e^{-\lambda}}\|y\|_{\check M}+q\|x\|_\infty,
\]
and thus $\mathcal T$ is well defined. Moreover, the same argument gives that 
\[
\|\mathcal Tx_1-\mathcal Tx_2\|_\infty \le q\|x_1-x_2\|_\infty, \quad \text{for $x_i\in \check C$, $i=1, 2$,}
\]
as
\[
(\mathcal Tx_1)(t)-(\mathcal Tx_2)(t)=\int_{-\infty}^\infty \mathcal G(t, s)B(s)(x_1(s)-x_2(s))\, ds, \quad t\in \R.
\]
Due to~\eqref{rob} we conclude that $\mathcal T$ is a contraction on $\check C$. Therefore, there exists a unique $x\in \check C$ that satisfies $\mathcal Tx=x$, that is,
\[
x(t)=\int_{-\infty}^\infty \mathcal G(t, s)(y(s)+B(s)x(s))\, ds, \quad t\in \R.
\]
Then
\[
\begin{split}
x(t)-T(t, r)x(r) &=\int_{-\infty}^t T(t, s)P(s)(y(s)+B(s)x(s))\, ds \\
&\phantom{=}-T(t, r)\int_{-\infty}^r T(r,s)P(s)(y(s)+B(s)x(s))\, ds \\
&\phantom{=}-\int_t^{\infty}T(t, s)(\Id-P(s))(y(s)+B(s)x(s))\, ds\\
&\phantom{=}+T(t, r)\int_r^\infty T(r, s)(\Id-P(s))(y(s)+B(s)x(s))\, ds\\
&=\int_r^t T(t, s)P(s)(y(s)+B(s)x(s))\, ds \\
&\phantom{=}+\int_r^tT(t, s)(\Id-P(s))(y(s)+B(s)x(s))\, ds \\
&=\int_r^t T(t, s)(y(s)+B(s)x(s))\, ds,
\end{split}
\]
for $t\ge r$.  This implies that $x\in \mathcal D(R)$ and $Rx=y+B(\cdot)x(\cdot)$. Due to~\eqref{RS}, we conclude that $x\in \mathcal D(S)$ and $Sx=y$. 

We have proved that for each $y\in \check M$ there exists $x\in \mathcal D(S)$ satisfying $Sx=y$. Note that such $x$ is also unique. In fact, assuming that $\tilde x\in \mathcal D(S)$ is such that $S\tilde x=y$, from~\eqref{RS} we conclude that
$R\tilde x=y+B(\cdot)\tilde x(\cdot)$. Observe that $y+B(\cdot)\tilde x(\cdot)\in \check M$. This implies (recall~\eqref{adm2} and the preceding discussion) that $\tilde x$ is a fixed point of $\mathcal T$ and, consequently, $\tilde x=x$.

From the above, we conclude that for each $y\in \check M$ there is a unique $x\in \check C$ such that $Sx=y$.
By~\cite[Theorem 5.8]{BDVbook}, we conclude that $U(t, s)$ admits an exponential dichotomy with respect to the norms $\|\cdot \|_t$, $t\in \R$ in the sense of~\cite[Definition 5.6]{BDVbook}, which together with~\eqref{ln1} and~\cite[Proposition 5.6]{BDVbook} implies that $U(t,s)$ admits a nonuniform exponential dichotomy. This completes the proof of the theorem.
\end{proof}

\begin{remark}
If $T(t, s)$ admits a nonuniform exponential dichotomy with respect to projections $P(t)=\Id$, $t\in \R$ or with $P(t)=0$, $t\in \R$, the condition~\eqref{rob} can be relaxed as follows:
\[
K\sup_{t\in \R}\int_{-\infty}^\infty e^{-\alpha |t-s|}b(s)\, ds<1. 
\]
This follows from the observation given in Remark~\ref{rem} (ii).
\end{remark}

\begin{remark}\label{RM3}
Let us compare Theorem~\ref{robthm} with~\cite[Theorem 9]{PPX}. Firstly, we note that~\cite[Theorem 9]{PPX} deals with the case of ordinary differential equations. More precisely, $T(t, s)$ is an evolution family corresponding to a nonautonomous differential equation
\[
x'=A(t)x \quad t\in \R,
\]
where $A\colon \R \to B(X)$ is a continuous map. In this case, $U(t, s)$ in~\eqref{U} is an evolution family associated with the equation \[ x = (A (t) + B (t)) x, \quad t\in \R.
\]
On the other hand, evolution families considered in our work are not necessarily associated with ordinary differential equations and generally consist of non-invertible operators. 

Secondly, \cite[Theorem 9]{PPX} established the robustness result under the following conditions (see~\cite[p. 408]{PPX}):
\begin{enumerate}
\item[(a)] $\varepsilon <2\alpha_2$;
\item[(b)] 
\begin{equation}\label{rob2}
K\sup_{t\in \R}\int_{-\infty}^\infty e^{-\alpha_1 |t-s|} \bar b(s)\, ds \le \frac{1}{K+1},
\end{equation}
where $\bar b(t):=\|B(t)\|e^{2\varepsilon |t|}$, $t\in \R$. Moreover, $\alpha_i>0$, $i=1, 2$ are such that $\alpha=\alpha_1+\alpha_2$.
\end{enumerate}
In contrast, we note that in the statement of Theorem~\ref{robthm} we do not make any assumptions related to the smallness of $\varepsilon$ compared to $\alpha$. In fact, we do not impose any requirement of this kind. Consequently, $(a)$ does not have its counterpart in the present work.
Moreover, assuming that~\eqref{D1} and~\eqref{D2} hold with $K>1$ we note that~\eqref{rob2} implies~\eqref{rob} as
\[
\int_{-\infty}^\infty e^{-\alpha |t-s|}b(s)\, ds\le \int_{-\infty}^\infty e^{-\alpha_1 |t-s|} \bar b(s)\, ds,
\]
and $\frac{1}{K(K+1)}<\frac{1}{2K}$ (recall that $\alpha_1<\alpha$ and $K>1$).
Notice that in this case, the requirement~\eqref{rob} is considerably weaker as in~\eqref{rob} we have a factor $e^{-\alpha |t-s|}$ instead of $e^{-\alpha_1|t-s|}$ in~\eqref{rob2}, and moreover $\bar b(t)= b(t)e^{\varepsilon |t|}$ for $t\in \R$.
\end{remark}
Let us discuss a concrete example.
\begin{example}
We consider~\cite[Example 1]{PPX}. Let $X=\R^2$ and consider a linear differential equation (on $X$) given by
\begin{equation}\label{LDE}
x'=A(t)x \quad t\in \R, 
\end{equation}
where 
\[
A(t)=\begin{pmatrix}
-\omega -at \sin t & 0\\
0 &\omega +at\sin t
\end{pmatrix},
\]
where $\omega >a$. As observed in~\cite[p.409]{PPX}, \eqref{LDE} admits a nonuniform exponential dichotomy. Moreover, \eqref{D1} and~\eqref{D2} hold with $K=e^{2a}$, $\alpha=\omega-a$, $\varepsilon =2a$ and $P(t)(v_1, v_2)=(v_1, 0)$ for $(v_1, v_2)\in X$ and $t\in \R$. 

It follows from Young's inequality for convolutions that~\eqref{rob} is satisfied whenever
\begin{equation}\label{some}
\frac 2 p (2K)^p \|b\|_{L^q}^p<\alpha,
\end{equation}
where $q>1$, $\frac 1 p+\frac 1 q=1$ and $\|b\|_{L^q}=\left (\int_{-\infty}^\infty b(t)^q\, dt \right )^{1/q}$. To compare~\eqref{some} with~\cite{PPX}, we first need to (slightly) correct~\cite[(6.3)]{PPX} which should read as follows:
\begin{equation}\label{some2}
\frac 2 p K^p (K+1)^p \|\bar b\|_{L^q}^p<\alpha_1 <\alpha.
\end{equation}
Observe that~\eqref{some2} implies~\eqref{some}, while the converse does not hold. Moreover, to apply the robustness result from~\cite{PPX} we would need to impose the requirement that $\omega >2a$ (see~\cite[p.409]{PPX}), which here we do not require. 
\end{example}

\section*{Acknowledgement}

 This paper has been funded by the European Union – NextGenerationEU-Statistical properties of random dynamical systems and other contributions to mathematical analysis and probability theory.

\end{document}